\documentclass{article}
\usepackage[utf8]{inputenc}

\usepackage{amsmath,amssymb,amsthm}
\usepackage{bm}
\usepackage{algorithm,algorithmic}
\usepackage{enumerate}
\usepackage{color}
\usepackage[margin=25mm]{geometry}
\usepackage[numbers]{natbib}
\usepackage{mathtools}
\mathtoolsset{showonlyrefs,showmanualtags}

\newtheorem{theorem}{Theorem}

\newtheorem{lemma}{Lemma}

\newtheorem{example}{Example}
\newtheorem{proposition}{Proposition}

\title{Some Martingale Properties of Simple Random Walk\\ and Its Maximum Process}
\author{Takahiko Fujita\thanks{Department of Data Science for Business Innovation, Chuo University, Japan}
\quad
Shotaro Yagishita\thanks{Department of Industrial and Systems Engineering, Chuo University, Japan, E-mail: a15.fjng@g.chuo-u.ac.jp}
\quad
Naohiro Yoshida\thanks{Department of Management, Tokyo University of Science, Japan}}
\date{\today}

\begin{document}

\maketitle

\begin{abstract}
In this paper, martingales related to simple random walks and their maximum process are investigated.
First, a sufficient condition under which a function with three arguments, time, the random walk, and its maximum process becomes a martingale is presented, and as an application, an alternative way of deriving the Kennedy martingale is provided.
Then, a complete characterization of a function with two arguments, the random walk and its maximum, being a martingale is presented. This martingale can be regarded as a discrete version of the Az\'ema--Yor martingale.
As applications of discrete Az\'ema--Yor martingale, a proof of the Doob's inequalities is provided and a discrete Az\'ema--Yor solution for the Skorokhod embedding problem for the simple random walk is formulated and examined in detail.
\end{abstract}

\section{Introduction}
For the simple random walk, there are various known properties that can be considered as an analogy of those for the Brownian motion \citep{csaki1981excursion,hu2003lengths,csorgHo1985strong,fujita2007remarkable,fujita2014one}.
Some properties are similarly investigated for the maximum process of the random walk and the random walk itself.
For instance, \citet{csaki1983combinatorial,simons1983discrete} and \citet{fujita2008arandom} provided discrete versions of the L\'evy's theorem, i.e., they showed that the maximum process and the local time they each defined has (nearly) the same probability distribution.
In particular, \citet {simons1983discrete} and \citet{fujita2008arandom} succeeded in giving the exact equality by properly defining the local time of the random walk.
\citet{pitman1975one, miyazaki1989theorem} and \citet{tanaka1989time} proved that $2M_\bullet-Z_\bullet$ is a Markov chain, where $Z_\bullet$ is the random walk and $M_\bullet$ denotes the maximum process of the random walk precisely defined below.
This result is called the Pitman's theorem for random walks and it is the discrete--time analogue of the similar theorem for the Brownian motion, which is also called the Pitman's theorem derived in \citep{pitman1975one} as a scaling limit of the random walk's one.
On the other hand, to the best of our knowledge, there are no studies dealing with martingale properties with respect to the random walk and its maximum process, except for \citet{fujita2014discrete}, who consider pricing of a lookback option.

In this paper, we investigate martingales which are functions of time, the random walk, and its maximum, and also show their applications.
Those results can be regarded as discrete analogues of those for the Brownian motion.
We present a sufficient condition for a function of time, the random walk, and its maximum to be a martingale in the next section.
As its application, a derivation of the Kennedy martingale different from the original one by \citet{kennedy1976some} is provided.
In Section \ref{sec:A-Y-martingale}, we present a complete characterization of a function of the random walk and its maximum being a martingale.
The martingales can be regarded as a discrete version of the Az\'ema--Yor martingales \citep{azema1979une,obloj2006local,obloj2006complete}.
We show the Doob's inequalities and a discrete Az\'ema--Yor solution for the Skorokhod embedding problem for the simple random walk by using the discrete Az\'ema--Yor martingales.

In the following, we prepare the notation that is used in this paper.
The set of nonnegative integers and positive integers are denoted by $\mathbb{Z}_{\ge0}$ and $\mathbb{Z}_{>0}$, respectively.
Let $\xi_1,\xi_2,\xi_3,\dots$ be a sequence of independent and identically distributed random variables on some probability space $(\Omega,\mathcal{F},P)$ such that
\begin{align}
    P(\xi_1=1)=p,\quad P(\xi_1=-1)=q,\quad P(\xi_1=0)=r,
\end{align}
where $p>0,~q>0$ and $p+q+r=1$.
A random walk $Z_\bullet=(Z_t)_{t=0}^\infty$ is defined by
\begin{align}
    Z_0=0,\quad Z_{t+1}=Z_t+\xi_{t+1}
\end{align}
for $t\in\mathbb{Z}_{\ge0}$.
We also define the maximum process of the random walk by
\begin{align}
    M_t = \max_{0\leq s\leq t} Z_s
\end{align}
for $t\in\mathbb{Z}_{\ge0}$.
Throughout this paper, we always consider the natural filtration $\mathcal{F}_t=\sigma(Z_s : s\le t)$.
The least integer greater than or equal to $x\in\mathbb{R}$ is denoted by $\lceil x\rceil$.
The following theorem will be used later.

\begin{theorem}[Martingale representation theorem for the simple random walk $\mbox{\citep[p.69]{fujita2008random}}$]\label{thm:martingale-representation}
Let $U_\bullet=(U_t)_{t=0}^\infty$ be a stochastic process adapted to $\mathcal{F}_t$.
The process $U_\bullet$ is a $\mathcal{F}_t$-martingale if and only if there exist  $\mathcal{F}_{t-1}$-measurable random variables $f_t, g_t$ such that
\begin{align}
    U_t-U_{t-1}=f_t\{\xi_t-(p-q)\}+g_t\{\xi_t^2-(p+q)\}
\end{align}
for all $t\in\mathbb{Z}_{>0}$.
\end{theorem}

\section{Sufficient condition for $\{f(t,M_t-Z_t,M_t)\}_{t=0}^\infty$ to be a martingale}\label{sec:A-Y-equation}
In this section, we present a difference equation of $f(t,x,y)$ that is a sufficient condition for $f(t,M_t-Z_t,M_t)$ to be a martingale.
The difference operators $\Delta_t^-, \Delta_x^+, \Delta_x^-, \Delta_y^+$ are defined by
\begin{align}
    \Delta_t^- f(t,x,y) &= f(t,x,y)-f(t-1,x,y),\quad\Delta_x^+ f(t,x,y) = f(t,x+1,y)-f(t,x,y),\\ 
    \Delta_x^- f(t,x,y) &= f(t,x,y) - f(t,x-1,y),\quad\Delta_y^+ f(t,x,y) = f(t,x,y+1)-f(t,x,y),
\end{align}
respectively.
Then we obtain the following proposition.

\begin{proposition}\label{prop:A-Y-eq}
Suppose that $f:\mathbb{Z}^3_{\geq 0} \to \mathbb{R}$ satisfies the difference equations
\begin{align}\label{eq:A-Y-eq-1}
    (p+q)\frac{1}{2}\Delta_x^+\Delta_x^-f(t,x,y)-(p-q)\frac{\Delta_x^++\Delta_x^-}{2}f(t,x,y)+\Delta_t^-f(t,x,y)=0 
\end{align}
for $t\geq 2,x\ge 1,y\geq 0$ and 
\begin{align}\label{eq:A-Y-eq-2}
    p\Delta_y^+ f(t,0,y)+q\Delta_x^+f(t,0,y)+\Delta_t^-f(t,0,y)=0
\end{align}
for $t\geq 1, y\geq 0$.
Then, $\{f(t,M_t-Z_t,M_t)\}_{t=0}^\infty$ is a martingale.
\end{proposition}

\begin{proof}
We can calculate for $t\ge1$ as
{\small \begin{align}
    &f(t,M_t-Z_t,M_t)-f(t-1,M_{t-1}-Z_{t-1},M_{t-1})\\
    =&\bm{1}_{\{M_{t-1}-Z_{t-1}>0\}}\bigg[\frac{-\Delta_x^+-\Delta_x^-}{2}f(t,M_{t-1}-Z_{t-1},M_{t-1})\xi_t+\frac{\Delta_x^+\Delta_x^-}{2}f(t,M_{t-1}-Z_{t-1},M_{t-1})\xi_t^2\bigg]\\
    &+\bm{1}_{\{M_{t-1}-Z_{t-1}=0\}}\bigg[\frac{\Delta_y^+-\Delta_x^+}{2}f(t,0,M_{t-1})\xi_t+\frac{\Delta_y^++\Delta_x^+}{2}f(t,0,M_{t-1})\xi_t^2\bigg]+\Delta_t^-f(t,M_{t-1}-Z_{t-1},M_{t-1})\\
    =&\bigg[\bm{1}_{\{M_{t-1}-Z_{t-1}>0\}}\frac{-\Delta_x^+-\Delta_x^-}{2}f(t,M_{t-1}-Z_{t-1},M_{t-1})+\bm{1}_{\{M_{t-1}-Z_{t-1}=0\}}\frac{\Delta_y^+-\Delta_x^+}{2}f(t,0,M_{t-1})\bigg]\{\xi_t-(p-q)\}\\
    &+\bigg[\bm{1}_{\{M_{t-1}-Z_{t-1}>0\}}\frac{\Delta_x^+\Delta_x^-}{2}f(t,M_{t-1}-Z_{t-1},M_{t-1})+\bm{1}_{\{M_{t-1}-Z_{t-1}=0\}}\frac{\Delta_y^++\Delta_x^+}{2}f(t,0,M_{t-1})\bigg]\{\xi_t^2-(p+q)\}\\
    &+\bm{1}_{\{M_{t-1}-Z_{t-1}>0\}}\bigg[(p+q)\frac{1}{2}\Delta_x^+\Delta_x^-f(t,M_{t-1}-Z_{t-1},M_{t-1})+(p-q)\frac{-\Delta_x^+-\Delta_x^-}{2}f(t,M_{t-1}-Z_{t-1},M_{t-1})\\
    &\quad+\Delta_t^-f(t,M_{t-1}-Z_{t-1},M_{t-1})\bigg]\\
    &+\bm{1}_{\{M_{t-1}-Z_{t-1}=0\}}\bigg[p\Delta_y^+f(t,0,M_{t-1})+q\Delta_x^+f(t,0,M_{t-1})+\Delta_t^-f(t,0,M_{t-1})\bigg]
\end{align}}
by simple algebraic computation.
From the assumption and Theorem \ref{thm:martingale-representation}, $\{f(t,M_t-Z_t,M_t)\}_{t=0}^\infty$ is a martingale.
\end{proof}

This is a discrete analogue of \citet[Proposition 2.1]{azema1979une} and \citet[Proposition 2]{nguyen2005some}.

\subsection{Derivation of Kennedy martingale}
By using Proposition \ref{prop:A-Y-eq}, we derive the Kennedy martingale given by \citet[Example 1]{kennedy1976some} in the different way.
Suppose that $f(t,x,y)=a^yb^th(x)$ satisfies the difference equations \eqref{eq:A-Y-eq-1}--\eqref{eq:A-Y-eq-2}, where $h:\mathbb{Z}_{\ge0}\to\mathbb{R}$ and $a, b\neq0$.
That is, it holds that
\begin{align}
    \frac{p+q}{2}a^yb^t\{h(x+1)-2h(x)+h(x-1)\}-\frac{p-q}{2}a^yb^t\{h(x+1)-h(x-1)\}+a^yb^th(x)(1-b^{-1})=0 
\end{align}
for any $t\geq 2,x\ge 1,y\geq 0$ and 
\begin{align}
    pa^yb^th(0)(a-1)+qa^yb^t\{h(1)-h(0)\}+a^yb^th(0)(1-b^{-1})=0
\end{align}
for any $t\geq 1, y\geq 0$, which are equivalent to
\begin{align}
    &qh(x+1)+(r-b^{-1})h(x)+ph(x-1)=0,\label{eq:twice-diff}\\
    &h(1)=\frac{b^{-1}-r-pa}{q}h(0)\label{eq:initial-value}
\end{align}
for any $x\ge1$.
Assuming further that $(r-b^{-1})^2-4pq>0$ and letting
\begin{align}
    \alpha_+=\frac{-(r-b^{-1})+\sqrt{(r-b^{-1})^2-4pq}}{2q},\quad\alpha_-=\frac{-(r-b^{-1})-\sqrt{(r-b^{-1})^2-4pq}}{2q},
\end{align}
the general solution of \eqref{eq:twice-diff} is $A\alpha_+^x+B\alpha_-^x$, where $A,B\in\mathbb{R}$.
Furthermore, setting $h(0)=\alpha_+^{-1}-\alpha_-^{-1}$, we obtain from the initial condition \eqref{eq:initial-value}
\begin{align}
    h(x)=(a-\alpha_-^{-1})\alpha_+^x-(a-\alpha_+^{-1})\alpha_-^x.
\end{align}
As a result, we derive the martingale $\big[a^{M_t}b^t\{(a-\alpha_-^{-1})\alpha_+^{M_t-Z_t}-(a-\alpha_+^{-1})\alpha_-^{M_t-Z_t}\}\big]_{t=0}^\infty$, which coincides with the Kennedy martingale by appropriate change of variable.
By using the martingale property, we can evaluate the probability generating function of $(Z_\tau,\tau)$,
as
\begin{align}
    E[a^{Z_\tau}b^\tau]=\frac{(\alpha_+^{-1}-\alpha_-^{-1})a^{-n}}{(a-\alpha_-^{-1})\alpha_+^n-(a-\alpha_+^{-1})\alpha_-^n}
\end{align}
for $a, b\neq0$ such that $(a-\alpha_-^{-1})\alpha_+^n-(a-\alpha_+^{-1})\alpha_-^n\neq0$ and $|r-b^{-1}|>\sqrt{4pq}$, where $\tau=\inf\{t\mid M_t-Z_t=n\}$ and $n\in\mathbb{Z}_{>0}$ (See \citet[Section 2]{kennedy1976some}).
We note that \citet[Section 3]{nguyen2005some} derived the Kennedy martingale for Brownian motion with drift from partial differential equations similar to the equations \eqref{eq:A-Y-eq-1}--\eqref{eq:A-Y-eq-2}.

\section{Discrete Az\'ema--Yor martingale}\label{sec:A-Y-martingale}
We show a necessary and sufficient condition for $\{H(Z_t,M_t)\}_{t=0}^\infty$ to be a martingale and its application in this section.

\begin{proposition}\label{prop:A-Y-martingale}
For $H:\{(x,y)\in \mathbb{Z}^2\mid \max\{x,0\}\le y\}\to\mathbb{R}$, $\{H(Z_t,M_t)\}_{t=0}^\infty$ is a martingale if and only if there exists a function $F:\mathbb{Z}_{\geq 0}\to\mathbb{R}$ such that
\begin{align}
    H(x,y)=
    \begin{cases}
        F(y)-\{F(y+1)-F(y)\}(y-x), &p=q,\\
        F(y)-\{F(y+1)-F(y)\}\frac{(q/p)^{-(y-x)}-1}{1-q/p}, &p\neq q.\\
    \end{cases}
\end{align}
\end{proposition}

\begin{proof}
The ``if'' argument is easily derived from Proposition \ref{prop:A-Y-eq}.
If $\{H(Z_t,M_t)\}_{t=0}^\infty$ is a martingale, then we have
\begin{align}
H(Z_t,M_t)&=E[H(Z_{t+1},M_{t+1})\mid \mathcal{F}_t]\\
&=E[H(Z_t+\xi_{t+1},\max\{M_t,Z_t+\xi_{t+1}\})\mid \mathcal{F}_t]\\
&=\bm{1}_{\{Z_t=M_t\}}\{pH(Z_t+1,M_t+1)+qH(Z_t-1,M_t)+rH(Z_t,M_t)\}\\
&\quad +\bm{1}_{\{Z_t<M_t\}}\{pH(Z_t+1,M_t)+qH(Z_t-1,M_t)+rH(Z_t,M_t)\}
\end{align}
for all $t\ge0$.
Noticing that for any $x,y\in\mathbb{Z}$ such that $\max\{x,0\}\le y$, there exists $t\ge0$ such that $P(Z_t=x,M_t=y)>0$, we obtain that
\begin{align}\label{eq:for-A-Y-1}
    \begin{cases}
    H(x,y)=pH(x+1,y+1)+qH(x-1,y)+rH(x,y), \quad &x=y,\\
    H(x,y)=pH(x+1,y)+qH(x-1,y)+rH(x,y), \qquad~~ &x<y
    \end{cases}
\end{align}
for any $x,y\in\mathbb{Z}$ such that $\max\{x,0\}\le y$.
Letting $F(y)= H(y,y)$, the equations \eqref{eq:for-A-Y-1} are equivalent to
\begin{align}\label{eq:for-A-Y-2}
    \begin{cases}
    p\{F(y+1)-F(y)\}=q\{H(y,y)-H(y-1,y)\}, \qquad\quad~ &x=y,\\
    p\{H(x+1,y)-H(x,y)\}=q\{H(x,y)-H(x-1,y)\}, \quad &x<y.
    \end{cases}
\end{align}

When $p=q$, we see that
\begin{align}
H(x,y)&=H(y,y)-\sum_{k=x}^{y-1}\{H(k+1,y)-H(k,y)\}\\
&=H(y,y)-\sum_{k=x}^{y-1}\{H(y,y)-H(y-1,y)\}\\
&=H(y,y)-\{H(y,y)-H(y-1,y)\}(y-x)\\
&=F(y)-\{F(y+1)-F(y)\}(y-x),
\end{align}
where the second equality follows from the second equation of \eqref{eq:for-A-Y-2}, and the fourth one from the first equation of \eqref{eq:for-A-Y-2} and the definition of $F$.
We have the desired result.

On the other hand, when $p\neq q$, we see from the equations \eqref{eq:for-A-Y-2} and the definition of $F$ that
\begin{align}
H(x,y)&=H(y,y)-\sum_{k=x}^{y-1}\{H(k+1,y)-H(k,y)\}\\
&=H(y,y)-\sum_{k=x}^{y-1}\left(\frac{p}{q}\right)^{y-1-k}\{H(y,y)-H(y-1,y)\}\\
&=H(y,y)-\{H(y,y)-H(y-1,y)\}\frac{(q/p)^{-(y-x)+1}-q/p}{1-q/p}\\
&=F(y)-\{F(y+1)-F(y)\}\frac{(q/p)^{-(y-x)}-1}{1-q/p},
\end{align}
which completes the proof.
\end{proof}

We refer to $\{H(Z_t,M_t)\}_{t=0}^\infty$ as discrete Az\'ema--Yor martingale.
For the case of $p=q$, the stochastic process $\big[F(M_t)-\{F(M_t+1)-F(M_t)\}(M_t-Z_t)\big]_{t=0}^\infty$ is a discrete version of the Az\'ema--Yor martingale \citep[Corollary 2.2]{azema1979une} for continuous local martingale, and Proposition \ref{prop:A-Y-martingale} is a discrete analogue of the complete characterization of a function of continuous local martingale and its maximum process being a local martingale \citep[Theorem 1]{obloj2006complete}.
The proof of the complete characterization in the continuous time setting is technical and not easy, but the one for the random walk is very simple.

\subsection{Doob's inequalities for simple random walk}
The following lemma is used below to derive the Doob's inequalities.

\begin{lemma}\label{lem:p-q-ineq}
Let $g_{p,q}:\mathbb{R}\to\mathbb{R}$ be the function
\begin{align}
    g_{p,q}(z)=
    \begin{cases}
        \qquad z, &p=q,\\
        \frac{(q/p)^{-z}-1}{1-q/p}, &p\neq q.\\
    \end{cases}
\end{align}
If $p\ge q$, then $g_{p,q}(z)\ge z$ holds for any $z\ge0$.
On the other hand, if $p\le q$, then $g_{p,q}(z)\le z$ holds for any $z\ge0$.
\end{lemma}

The proof is not difficult, but is included in \ref{sec:proof} for completeness.
We first show the Doob's maximal inequality for the random walk.

\begin{proposition}\label{prop:Doob-maximal}
Let $\lambda$ be a positive real number.
\begin{enumerate}[(i)]
    \item If $p\ge q$ (i.e., $Z_\bullet$ is a submartingale), then the inequality
    \begin{align}\label{eq:Doob-maximal}
        \lambda P(M_t\ge\lambda)\le\lceil\lambda\rceil P(M_t\ge\lambda)\le E[\bm{1}_{\{M_t\ge\lambda\}}Z_t]
    \end{align}
    holds.
    \item If $p\le q$, then the inequality
    \begin{align}\label{eq:inverse-Doob-maximal}
        E[\bm{1}_{\{M_t\ge\lambda\}}Z_t]\le\lceil\lambda\rceil P(M_t\ge\lambda)
    \end{align}
    holds.
\end{enumerate}
Especially, the equality $\lceil\lambda\rceil P(M_t\ge\lambda)=E[\bm{1}_{\{M_t\ge\lambda\}}Z_t]$ holds if $p=q$.
\end{proposition}

\begin{proof}
Let $F:\mathbb{Z}_{\geq 0}\to\mathbb{R}$ be the function
\begin{align}
    F(y)=\bm{1}_{\{y\ge\lceil\lambda\rceil\}}(y-\lceil\lambda\rceil),
\end{align}
then we obtain that
\begin{align}
    F(y+1)-F(y)=\bm{1}_{\{y\ge\lceil\lambda\rceil\}}.
\end{align}
We see from Proposition \ref{prop:A-Y-martingale} that $U_\bullet^F$ defined by
\begin{align}\label{eq:martingale-for-maximal}
    \begin{split}
        U_t^F &=F(M_t)-\{F(M_t+1)-F(M_t)\}g_{p,q}(M_t-Z_t)\\ &=\bm{1}_{\{M_t\ge\lceil\lambda\rceil\}}(M_t-\lceil\lambda\rceil)-\bm{1}_{\{M_t\ge\lceil\lambda\rceil\}}g_{p,q}(M_t-Z_t)
    \end{split}
\end{align}
is a martingale with initial value of $0$, so we have
\begin{align}\label{eq:mean-for-maximal}
    E[U_t^F]=0.
\end{align}

If $p\ge q$, the first inequality of \eqref{eq:Doob-maximal} is trivial.
From the equation \eqref{eq:mean-for-maximal}, Lemma \ref{lem:p-q-ineq}, and $\{M_t\ge\lceil\lambda\rceil\}=\{M_t\ge\lambda\}$, it holds that
\begin{align}\label{eq:evaluation-maximal}
    \begin{split}
        0=E[U_t^F] &\le E[\bm{1}_{\{M_t\ge\lceil\lambda\rceil\}}(M_t-\lceil\lambda\rceil)-\bm{1}_{\{M_t\ge\lceil\lambda\rceil\}}(M_t-Z_t)]\\
        &=E[\bm{1}_{\{M_t\ge\lceil\lambda\rceil\}}Z_t]-\lceil\lambda\rceil P(M_t\ge\lceil\lambda\rceil)\\
        &=E[\bm{1}_{\{M_t\ge\lambda\}}Z_t]-\lceil\lambda\rceil P(M_t\ge\lambda).
    \end{split}
\end{align}

If $p\le q$, we have
\begin{align}\label{eq:evaluation-inverse-maximal}
    \begin{split}
        0=E[U_t^F] &\ge E[\bm{1}_{\{M_t\ge\lceil\lambda\rceil\}}(M_t-\lceil\lambda\rceil)-\bm{1}_{\{M_t\ge\lceil\lambda\rceil\}}(M_t-Z_t)]\\
        &=E[\bm{1}_{\{M_t\ge\lceil\lambda\rceil\}}Z_t]-\lceil\lambda\rceil P(M_t\ge\lceil\lambda\rceil)\\
        &=E[\bm{1}_{\{M_t\ge\lambda\}}Z_t]-\lceil\lambda\rceil P(M_t\ge\lambda)
    \end{split}
\end{align}
from the equation \eqref{eq:mean-for-maximal}, Lemma \ref{lem:p-q-ineq}, and $\{M_t\ge\lceil\lambda\rceil\}=\{M_t\ge\lambda\}$.
\end{proof}

Next, we show the Doob's $L^p$ inequality for the random walk.

\begin{proposition}\label{prop:Doob-lp}
Let $\pi>1$ and $p\ge q$ (i.e., $Z_\bullet$ is a submartingale).
It hold that
\begin{align}
    E[M_t^\pi]\le\left(\frac{\pi}{\pi-1}\right)^\pi E[|Z_t|^\pi]
\end{align}
\end{proposition}

\begin{proof}
Let $F:\mathbb{Z}_{\geq 0}\to\mathbb{R}$ be the function
\begin{align}
    F(0)=0,\quad F(y+1)=F(y)+\pi y^{\pi-1},
\end{align}
then we see that
\begin{align}\label{eq:power-p}
    F(y)=\sum_{k=0}^{y-1}\pi k^{\pi-1} \le\sum_{k=0}^{y-1}\int_k^{k+1} \pi t^{\pi-1} dt=\int_0^y \pi t^{\pi-1} dt=y^\pi.
\end{align}
Let
\begin{align}\label{eq:martingale-for-lp}
    \begin{split}
        U_t^F &=F(M_t)-\{F(M_t+1)-F(M_t)\}g_{p,q}(M_t-Z_t)\\
        &=F(M_t)-\pi M_t^{\pi-1}g_{p,q}(M_t-Z_t),
    \end{split}
\end{align}
then from Proposition \ref{prop:A-Y-martingale}, $U_\bullet^F$ is a martingale with initial value of $0$.
Since the inequality $U_t^F\le M_t^\pi-\pi M_t^{\pi-1}(M_t-Z_t)$ holds by Lemma \ref{lem:p-q-ineq} and the inequality \eqref{eq:power-p}, we have
\begin{align}
    0\le E[M_t^\pi-\pi M_t^{\pi-1}(M_t-Z_t)]=(1-\pi)E[M_t^\pi]+\pi E[M_t^{\pi-1}Z_t].
\end{align}
By using the H\"older's inequality, we have the desired result.
\end{proof}

\citet[Section 3]{obloj2006local} derived the Doob's inequalities for general discrete--time submartingales in a manner similar to proofs of Propositions \ref{prop:Doob-maximal} and \ref{prop:Doob-lp} with $p=q$ by using their discrete balayage formula.
On the other hand, Proposition \ref{prop:Doob-maximal} leads to the maximal ``equality'' for symmetric random walk.

\subsection{Discrete Az\'ema--Yor solution for Skorokhod embedding problem}
In the following, we present a discrete Az\'ema--Yor solution for the Skorokhod embedding problem for the simple random walk by using the discrete Az\'ema--Yor martingales.
The Skorokhod embedding problem for standard Brownian motion \citep{skorokhod1965studies} says, ``For given a centered probability measure $\mu$ on $\mathbb{R}$, find a stopping time $T$ such that the stopped process $(B_{t\wedge T})_{t=0}^\infty$ is uniformly integrable and $T$ embeds $\mu$ in $B_\bullet$, i.e., $B_T\sim\mu$,'' where $(B_t)_{t=0}^\infty$ is the Brownian motion.
Unlike the case of the Brownian motion, there exists a centered probability measure on $\mathbb{Z}$ that can not be embedded such that the stopped process is uniformly integrable in the case of the simple symmetric random walk \citep[Proposition 1]{cox2008classes}.
The following proposition shows that some measures are embedded in the symmetric random walk.

\begin{proposition}\label{prop:A-Y-solution}
Let $p=q$ and $\mu$ be a centered probability measure on $\mathbb{Z}$.
We define a function $\psi_\mu :\mathbb{Z}\to \mathbb{R}$ by
\begin{align}
    \psi_\mu(x)=
    \begin{cases}
        x+\frac{\mu(\{x+1,x+2,\dots\})}{\mu(\{x\})}, &x\in\mathrm{supp}\,\mu,\\
        \qquad-1, &x \not\in \mathrm{supp}\,\mu,
    \end{cases}
\end{align}
where $\mathrm{supp}\,\mu=\{x\in\mathbb{Z}\mid\mu(\{x\})\not=0\}$.
Suppose further that the following assumptions hold:
\begin{enumerate}[({A}1)]
    \item It holds that $\psi_\mu(\mathrm{supp}\,\mu)\subset\mathbb{Z}_{\ge0}$;
    \item The function ${\psi_\mu}_{|\mathrm{supp}\,\mu}$ is strictly increasing.
\end{enumerate}
Let $T_\mu=\inf\{t\mid M_t=\psi_\mu(Z_t)\}$.
Then, $E[T_\mu]<\infty$, $(Z_{t\wedge T_\mu})_{t=0}^\infty$ is uniformly integrable, and the stopping time $T_\mu$ embeds $\mu$ in $Z_\bullet$, that is, $Z_{T_\mu}\sim\mu$ holds.
\end{proposition}

\begin{proof}
We will show $E[T_\mu]<\infty$ and the uniformly integrability of $(Z_{t\wedge T_\mu})_{t=0}^\infty$ at first.
Since $\mathrm{supp}\,\mu$ is bounded below from the assumptions (A1) and (A2), $\mathrm{supp}\,\mu$ can be represented as $\{x_0,x_1,\ldots\}$ or $\{x_0,x_1,\ldots,x_m\}$ by a strictly increasing sequence $x_0<x_1<\cdots$ or $x_0<x_1<\cdots<x_m$.
It holds that $\psi_\mu(x)\mu(\{x\})=x\mu(\{x\})+\mu(\{x+1,x+2,\dots\})$ for all $x\in\mathrm{supp}\,\mu$ by the definition of $\psi_\mu$, so we have
\begin{align}
    \sum_{i\ge0}\psi_\mu(x_i)\mu(\{x_i\}) &=\sum_{i\ge0}x_i\mu(\{x_i\})+\sum_{i\ge0}\mu(\{x_i+1,x_i+2,\dots\})\\
    &=\sum_{i\ge0}\sum_{j>i}\mu(\{x_j\})\\
    &=\sum_{j\ge0}\sum_{i=0}^{j-1}\mu(\{x_j\})\\
    &=\sum_{j\ge0}j\mu(\{x_j\})\\
    &=\sum_{i\ge0}i\mu(\{x_i\}),
\end{align}
where the second equality follows from $\sum_{i\ge0}x_i\mu(\{x_i\})=0$.
Noting that $\sum_{i\ge0}\mu(\{x_i+1,x_i+2,\dots\})<\infty$ and it holds that $\psi_\mu(x_i)\ge i$ from the assumptions (A1) and (A2), we obtain that
\begin{align}\label{eq:psi-value}
    \psi_\mu(x_i)=i.
\end{align}
The function $\psi_\mu(x)-x$ is nonnegative on $\mathrm{supp}\,\mu$ by the definition of $\psi_\mu$ and is decreasing on $\mathrm{supp}\,\mu$ because we can evaluate as
\begin{align}
    \psi_\mu(x_{i+1})-x_{i+1}-\{\psi_\mu(x_i)-x_i\}=1+x_i-x_{i+1}\le1-1=0.
\end{align}
In particular, $C\coloneqq\psi_\mu(x_0)-x_0\ge\psi_\mu(x)-x$ for all $x\in\mathrm{supp}\,\mu$.
Let $T_1=\inf\{t\ge C\mid Z_t-Z_{t-C}=-C\}$ and $T_2=|\{t\mid t\le T_1,~ Z_t-Z_{t-1}\in\{0,1\}\}|$, where the cardinality of a set $A$ is denoted by $|A|$.
We have $T_1\le (C-1)T_2+C+T_2=C(T_2+1)$ because $\sup\{c\mid t\le T_1,~ Z_t-Z_{t-1}\in\{0,1\},~ Z_{t-1}-Z_{t-1-c}=-c\}\le C-1$.
Since $T_2$ can be regarded as the number of failures before the first success in independent Bernoulli trials with probability of success $(p+r)^C$, it holds that $E[T_2]<\infty$, which implies that $T_1<\infty~a.s.$
To show $T_\mu\le T_1~a.s.$, we divide the case into (i) $\mathrm{supp}\,\mu$ is finite. i.e., $\mathrm{supp}\,\mu=\{x_0,x_1,\ldots,x_m\}$ and $M_{T_1}(\omega)>m$, and (ii) otherwise.

\noindent Case (i):
Note that there exists $s\le T_1(\omega)$ such that $M_s(\omega)=Z_s(\omega)=m$.
The equality $m=\psi_\mu(x_m)=x_m$ holds by the definition of $\psi_\mu$ and the equality \eqref{eq:psi-value}, so $M_s(\omega)=m=\psi_\mu(m)=\psi_\mu(Z_s(\omega))$, which implies that $T_\mu(\omega)\le s\le T_1(\omega)$.

\noindent Case (ii):
In this case, there exists $x\in\mathrm{supp}\,\mu$ such that $\psi_\mu(x)=M_{T_1}(\omega)$.
Since we can evaluate as
\begin{align}
    Z_{T_1}(\omega)\le M_{T_1}(\omega)-C\le M_{T_1}(\omega)-\psi_\mu(x)+x=x\le\psi_\mu(x)=M_{T_1}(\omega),
\end{align}
there exists $s\le T_1(\omega)$ such that $M_s(\omega)=M_{T_1}(\omega)$ and $Z_s(\omega)=x$, so $M_s(\omega)=\psi_\mu(x)=\psi_\mu(Z_s(\omega))$, which implies that $T_\mu(\omega)\le s\le T_1(\omega)$.

\noindent As a result, $T_\mu\le T_1<\infty~a.s.$ and hence $E[T_\mu]<\infty$.
Since $|Z_{t\wedge T_\mu}|\le t\wedge T_\mu\le T_\mu$, $(Z_{t\wedge T_\mu})_{t=0}^\infty$ is uniformly integrable by Theorem 27.2 in \cite{jacod2004probability}.


Next, we show $Z_{T_\mu}\sim\mu$.
For any $x\in\mathrm{supp}\,\mu$, let $F_x:\mathbb{Z}_{\geq 0}\to\mathbb{R}$ be the function
\begin{align}
    F_x(y)=\bm{1}_{\{y>\psi_\mu(x)\}}
\end{align}
then we see from Proposition \ref{prop:A-Y-martingale} that $U_\bullet^x$ defined by
\begin{align}\label{eq:martingale-for-skorokhod}
    \begin{split}
        U_t^x &=F_x(M_t)-\{F_x(M_t+1)-F_x(M_t)\}(M_t-Z_t)\\ &=\bm{1}_{\{M_t>\psi_\mu(x)\}}-\bm{1}_{\{M_t=\psi_\mu(x)\}}(M_t-Z_t)
    \end{split}
\end{align}
is a martingale with initial value of $0$ and $(U_{t\wedge T_\mu}^x)_{t=0}^\infty$ is also a martingale by the optional stopping theorem.
Since
\begin{align}
    \left|U_{t\wedge T_\mu}^x\right|\le\bm{1}_{\{M_{t\wedge T_\mu}>\psi_\mu(x)\}}+\bm{1}_{\{M_{t\wedge T_\mu}=\psi_\mu(x)\}}(M_{t\wedge T_\mu}-Z_{t\wedge T_\mu})\le1+t\wedge T_\mu\le1+T_\mu
\end{align}
we have
\begin{align}\label{eq:after-limit}
    E[\bm{1}_{\{M_{T_\mu}>\psi_\mu(x)\}}]=E[\bm{1}_{\{M_{T_\mu}=\psi_\mu(x)\}}(M_{T_\mu}-Z_{T_\mu})]
\end{align}
by the dominated convergence theorem.
Since $M_{T_\mu}=\psi_\mu (Z_{T_\mu})~a.s.$, we have
\begin{align}\label{eq:diffecence-on-supp}
    P(Z_{T_\mu}>x)=(\psi_\mu (x)-x)P(Z_{T_\mu}=x)
\end{align}
from the equality \eqref{eq:after-limit} and the assumption (A2).
Note that $Z_{T_\mu}\in\mathrm{supp}\,\mu~a.s.$ holds because $M_{T_\mu}=\psi_\mu (Z_{T_\mu})~a.s.$
We see from the equality \eqref{eq:diffecence-on-supp} and the definition of $\psi_\mu$ that
\begin{align}
    \frac{\mu(\{x_0+1,x_0+2,\dots\})}{\mu(\{x_0\})}P(Z_{T_\mu}=x_0)=P(Z_{T_\mu}>x_0)=1-P(Z_{T_\mu}=x_0),
\end{align}
that is,
\begin{align}
    P(Z_{T_\mu}=x_0)=\frac{1}{1+\frac{\mu(\{x_0+1,x_0+2,\dots\})}{\mu(\{x_0\})}}=\mu(\{x_0\}).
\end{align}
Suppose that $P(Z_{T_\mu}=x_j)=\mu(\{x_j\})$ for any $j$ such that $0\le j\le i-1$,
 then we have
\begin{align}
    \frac{\mu(\{x_i+1,x_i+2,\dots\})}{\mu(\{x_i\})}P(Z_{T_\mu}=x_i) &=P(Z_{T_\mu}>x_i)\\
    &=1-P(Z_{T_\mu}=x_i)-P(Z_{T_\mu}<x_i)\\
    &=1-P(Z_{T_\mu}=x_i)-\mu(\{\dots,x_i-2,x_i-1\})\\
    &=\mu(\{x_i,x_i+1,\dots\})-P(Z_{T_\mu}=x_i),
\end{align}
so the equality $P(Z_{T_\mu}=x_i)=\mu(\{x_i\})$ holds.
As a result, we obtain that $P(Z_{T_\mu}=x)=\mu(\{x\})$ for all $x\in\mathrm{supp}\,\mu$, which is the desired result.
\end{proof}

For the simple random walk with $p=q=\frac{1}{2}$, \citet{obloj2004skorokhod} and \citet{cox2008classes} studied a solution using the barycenter function directly as an analogue of Az\'ema--Yor solution in continuous--time setting, but $\psi_\mu$ in our discrete Az\'ema--Yor solution is a different function.
To the best of our knowledge, the integrability of the Az\'ema--Yor stopping time using the barycenter function has not been appeared in both continuous--time and discrete--time settings.
In the following, we present centered probability measures embedded in $Z_\bullet$ by our discrete Az\'ema--Yor stopping time.

\begin{example}[Centered geometric distribution]\label{ex:geometric}
Let $n\in\mathbb{Z}_{>0}$ and $\pi=\frac{1}{1+n}$.
We define a centered probability measure $\mu$ by
\begin{align}
    \mu(\{x\})=\pi(1-\pi)^{x+n}
\end{align}
for $x\ge-n$.
Then, we obtain that
\begin{align}
    \psi_\mu(x)=x+\frac{\mu(\{x+1,x+2,\dots\})}{\mu(\{x\})}=x+\frac{\pi(1-\pi)^{x+1+n}}{\pi(1-\pi)^{x+n}}=x+\frac{1-\pi}{\pi}=x+n,
\end{align}
so $T_\mu=\inf\{t\mid M_t=\psi_\mu(Z_t)\}=\inf\{t\mid M_t=Z_t+n\}$ embeds $\mu$ in $Z_\bullet$ from Proposition \ref{prop:A-Y-solution}.
\end{example}

\begin{example}[Discrete uniform distribution with interval $2$]\label{ex:uniform}
Let $n\in\mathbb{Z}_{>0}$.
We define a centered probability measure $\mu$ by 
\begin{align}
    \mu(\{x\})=\frac{1}{2n+1}
\end{align}
for $x\in\{-2n,-2n+2,\ldots,-2,0,2,\ldots,2n-2,2n\}$.
Then, we obtain that
\begin{align}
    \psi_\mu(x)=x+\frac{\mu(\{x+1,x+2,\dots\})}{\mu(\{x\})}=x+2n-i=x+2n-n-\frac{x}{2}=\frac{1}{2}x+n
\end{align}
for $x=-2n+2i,~ i=0,1,\ldots,2n-1,2n$, so $T_\mu=\inf\{t\mid M_t=\psi_\mu(Z_t)\}=\inf\{t\mid M_t=\frac{1}{2}Z_t+n\}$ embeds $\mu$ in $Z_\bullet$ from Proposition \ref{prop:A-Y-solution}.
\end{example}

\section{Conclusion}\label{sec:conclusion}
We have presented the martingales associated with the simple random walk and its maximum process, and their applications.
Interestingly, the stopping time in which the barycenter function is used \citep{obloj2004skorokhod,cox2008classes} and our discrete Az\'ema--Yor stopping time in which $\psi_\mu$ is used coincide in examples \ref{ex:geometric} and \ref{ex:uniform}, even though they are different functions in general.
The relationship between the sets of measures that can be embedded by each solution is interesting but not yet known.

\appendix
\def\thesection{Appendix \Alph{section}}

\section{Proof of Lemma \ref{lem:p-q-ineq}}\label{sec:proof}
\begin{proof}
Since the case $p=q$ is obvious, we first consider the case $p>q$.
In this case, the function $g_{p,q}$ is convex, so
\begin{align}\label{eq:convexity-g}
    g_{p,q}(z)\ge g_{p,q}(0)+g_{p,q}'(0)z=\frac{-\log(q/p)}{1-q/p}z
\end{align}
holds for any $z\in\mathbb{R}$.
From the convexity of $-\log(\cdot)$, we have
\begin{align}\label{eq:convexity--log}
    -\log(q/p)\ge-\log1-(q/p-1)=1-q/p,
\end{align}
so $\frac{-\log(q/p)}{1-q/p}\ge1$ holds from the inequality \eqref{eq:convexity--log} and $1-q/p>0$.
From this and the inequality \eqref{eq:convexity-g}, we see that
\begin{align}
    g_{p,q}(z)\ge\frac{-\log(q/p)}{1-q/p}z\ge z
\end{align}
for any $z\ge0$.

Next, we consider the case $p<q$, that is, $1-q/p<0$, so we have $\frac{-\log(q/p)}{1-q/p}\le1$ from the inequality \eqref{eq:convexity--log}.
In this case, the function $g_{p,q}$ is concave, so
\begin{align}\label{eq:concavity-g}
    g_{p,q}(z)\le g_{p,q}(0)+g_{p,q}'(0)z=\frac{-\log(q/p)}{1-q/p}z\le z
\end{align}
holds for any $z\ge0$.
\end{proof}

\bibliography{ref.bib}
\bibliographystyle{plainnat}

\end{document}